\documentclass[12pt]{article}
\usepackage{amssymb, amsmath}
\newcommand{\qed}{\hfill \rule{2.5mm}{2.5mm}}
\newcommand{\R}{{\mathbb R}}

\newcommand{\N}{{\mathbb N}}

\begin{document}
\newtheorem{thm}{Theorem}[section]
\newtheorem{defs}[thm]{Definition}
\newtheorem{note}[thm]{Note}
\newtheorem{lem}[thm]{Lemma}
\newtheorem{rem}[thm]{Remark}
\newtheorem{cor}[thm]{Corollary}
\newtheorem{prop}[thm]{Proposition}
\renewcommand{\theequation}{\arabic{section}.\arabic{equation}}
\newcommand{\newsection}[1]{\setcounter{equation}{0} \section{#1}}
\title{The H\'ajek-R\'enyi-Chow maximal inequality and a strong law of large numbers in Riesz spaces
      \footnote{{\bf Keywords:} Riesz spaces; vector lattices; maximal inequality; Clarkson's inequality; submartingale convergence; strong law of large numbers.\
      {\em Mathematics subject classification (2010):} 46B40; 60F15; 60F25.}}
\author{Wen-Chi Kuo\footnote{Supported in part by  National Research Foundation of South Africa grant number CSUR160503163733.},
David F. Rodda\footnote{Supported in part by National Research Foundation of South Africa grant number 110943.}\, \&
 Bruce A. Watson \footnote{Supported in part by the Centre for Applicable Analysis and
Number Theory and by National Research Foundation of South Africa grant IFR170214222646 with grant no. 109289.} \\
School of Mathematics\\
University of the Witwatersrand\\
Private Bag 3, P O WITS 2050, South Africa }
\maketitle
\begin{abstract}\noindent
In this paper we generalize the H\'ajek-R\'enyi-Chow maximal inequality for submartingales to $L^p$ type Riesz spaces with conditional expectation operators. As applications we obtain a submartingale convergence theorem
 and a strong law of large numbers in Riesz spaces. Along the way we develop a Riesz space variant of the Clarkson's inequality for $1\le p\le 2$. 
\end{abstract}
\parindent=0cm
\parskip=0.5cm

\newsection{Introduction}

In a Dedekind complete Riesz space $E$ with weak order unit, say $e$, we say that $T$ is a conditional expectation on $E$ if $T$ is a
positive order continuous linear projection on $E$ which maps weak order units to weak order units and has range, $R(T)$, a Dedekind complete Riesz subspace of $E$, see \cite{KLW2} for more details. It should be noted that the only conditional expectation 
operator which is also a band projection is the identity map.
A conditional expectation operator $T$ is said to be strictly positive if $T|f|=0$ implies that $f=0$.
Every Archimedean Riesz space $E$ can be extended uniquely (up to Riesz isomorphism) to a universally complete space $E^{u}$, see \cite{RieszSpacesIIZaanenBook}. 
It was shown in \cite{KLW2} that the domain of a strictly positive conditional expectation operator $T$ can be extended to its natural domain $L^{1}(T)$ in $E^{u}$. In particular $L^{1}(T)=dom(T)-dom(T)$ where $f\in dom(T)$ if $f\in E^u_+$, the positive cone of $E^u$, and there is an upwards directed net $f_\alpha$ in $E_+$ with the net $Tf_\alpha$ order bounded in $E^u$ and in this case the value assigned to $Tf$ is the order limit in $E^u$ of the net $Tf_\alpha$. 
Given that the above extensions can be made we will assume throughout that $T$ is a conditional expectation operator acting on $L^{1}(T)$. 
The space $E^{u}$ is an $f$-algebra with multiplication defined so that
the chosen weak order unit, $e$, is the algebraic unit.
Further, it was shown in \cite{KRW} that $R(T)=\{Tf|f\in L^{1}(T)\}$ is an $f$-algebra and $L^{1}(T)$ is an $R(T)$-module with $R(T)$-valued norm $\|f\|_{1,T}:=T|f|$.
The space $L^2(T)$ was introducted in \cite{LW} and generalized to $L^p(T) = \{x\in L^{1}(T): |x|^{p}\in L^{1}(T)\}, 1<p<\infty,$ in \cite{LpSpacesAz} where functional calculus was used to define $f(x)=x^p$ for $x\in E_+^u$. Much of the mathematical machinery needed to work in $L^p(T), 1<p<\infty,$ was developed in \cite{JensensGrobler} even though such spaces were not considered there. Again these spaces are $R(T)$-modules with associated 
$R(T)$-valued norms $\|f\|_{p,T}:=(T|f|^p)^{1/p}$. 
We note that in \cite{KRW} the $R(T)$-module $L^\infty(T)=\{x\in L^{1}(T): |x|\le y \mbox{ for some } y\in R(T)\}$ was considered, with $R(T)$-valued norm $\|f\|_{\infty,T}:=\inf \{ y\in R(T)_+ : |f|\le y\}$. Here we have that $L^\infty(T)$ is an $f$-algebra order dense in 
$L^1(T)$ and having $L^\infty(T)\subset L^p(T)\subset L^1(T)$ for all $1<p<\infty$. 
Regarding $L^p$ type spaces we also note the work of Boccuto, Candeloro and Sambucini in \cite{BCS}.
In Section 3, we generalize Clarkson's inequality to $L^p(T), 1\le p\le 2$.

A filtration on a Dedekind complete Riesz space $E$ with weak order unit is a family of conditional expectation operators $(T_{i})_{i\in\mathbb{N}}$ defined on $E$ having $T_{i}T_{j}=T_{j}T_{i}=T_{i}$ for all $i<j$. 
We say that a sequence of elements $(f_{i})_{i\in\mathbb{N}}$ in a Riesz space is adapted to a filtration 
$(T_{i})_{i\in\mathbb{N}}$ if $f_{i}\in {R}(T_{i})$ for all $i\in\N$. 
A sequence $(f_{i})_{i\in\mathbb{N}}$ is said to be predictable if $f_{i}\in {R}(T_{i-1})$ for each $i\in\mathbb{N}$. 
A Riesz space (sub, super) martingale is a double sequence $(f_{i},T_{i})_{i\in\mathbb{N}}$ with $(f_{i})_{i\in\mathbb{N}}$ adapted to the filtration  $(T_{i})_{i\in\mathbb{N}}$ and $T_{i}f_{j} (\geq, \leq)=f_{i}$ for $i<j$. The fundamentals of such processes can be found in \cite{KLW1,KLW2,KLW3} as well as their continuous time versions in \cite{JJG-1, JJG-2}. 
In Section 4, we give the H\'ajek-R\'enyi-Chow maximal inequality for Riesz space submartingales, see \cite[Theorem 1]{ChowIneqAndLLN} and \cite[Proposition (6.1.4)]{EAndSBook} for measure theoretic versions. 
The H\'ajek-R\'enyi-Chow maximal inequality for submartingales has as a special case Doob's maximal inequality. We note that maximal inequalities have been obtained for Riesz space positive supermartingales in \cite[Lemma 3.1]{QuadVarGrobler} and for Riesz space quasi-martingales in \cite[Theorem 6.2.10]{Vardy}. The H\'ajek-R\'enyi-Chow maximal inequality for submartingales is applied, in Theorem \ref{SubmartingaleConvergence}, to non-negative submartingales to obtain weighted convergence, via a proof which does not use of upcrossing. 
We note that this theorem can be deduced directly from \cite[Theorem 3.5]{KLW3}, which is, however, based on the Riesz space upcrossing theorem.
For $E$ a Dedekind complete Riesz space with weak order unit, $e$, and $(B_n)$ an increasing sequence of bands in $E$, with associated band projections $(P_n)$, it was proved in \cite{Stoica} that $x_n/b_n\to 0$, in order, as $n\to\infty$ if $x_n\in B_n$ with 
$P_nx_{n+1}=x_n$ and $|x_{n+1}-x_n|\le c_n e$, for all $n\in\N$.
Here $c_n>0$ and $0<b_n\uparrow\infty$,  for all $n\in\N$, with
$\displaystyle{\frac{1}{b_n}\left(\sum_{i=1}^n c_i^2\right)^{1/2}\to 0}$ as $n\to\infty$.
In Section 5, we conclude by giving Chow's strong law of large numbers in $L^p(T), 1<p<\infty$, see \cite{ChowIneqAndLLN, ChowLLNPBig} and \cite[Theorems 6.1.8 and 6.1.9]{EAndSBook} for measure versions. 

We note that, for Riesz space processes, a strong law of large numbers for ergodic processes was given in \cite{KLW4}, a weak law of large number for mixingales in \cite{KVW1} and Bernoulli's law of large numbers in \cite{KVW2}.

\section{Weighted Ces\`aro means}

In this section we give a version of Kronecker's Lemma for weighted Ces\`aro means in an Archimedean Riesz space.

\begin{lem}\label{CesaroConvergence}
 Let $E$ be an Archimedean Riesz space and $(s_{n})$ be a sequence in $E_+$ order convergent to $0$. If $b_{n}$ is a non-decreasing sequence of non-negative real numbers divergent to $+\infty$, then $\frac{1}{b_{n}}\sum_{i=1}^{n-1}(b_{i+1}-b_{i})s_{i}$ converges to zero in order as $n\to\infty$. 
\end{lem}

\begin{proof}  
 By the order convergence of $(s_{n})$ to $0$, there is sequence $(v_{n})$ in $E$ such that $s_{n}\leq v_{n}\downarrow 0$, for $n\in\mathbb{N}$.
 As $$0\le \frac{1}{b_{n}}\sum_{i=1}^{n-1}(b_{i+1}-b_{i})s_{i}\le \frac{1}{b_{n}}\sum_{i=1}^{n-1}(b_{i+1}-b_{i})v_{i}=:z_n,$$ it suffices to show that
 $z_n\to 0$ in order.
 For $n\in \N$, let $N_n:=\max\{j\in\N\,|\,j^2b_j\le b_n\}$, then $(N_n)$ is a non-decreasing sequence in $\N$ with $N_n\to\infty$ as $n\to\infty$ 
 and $N_n< n$ for $n\ge 2$.
 Now, for $n\ge 2$,
 \begin{eqnarray*}
 z_n&=&\frac{1}{b_n}\left(b_nv_{n-1}-b_1v_1+\sum_{i=2}^{N_n}b_i(v_{i-1}-v_i)+\sum_{i=N_n+1}^{n-1}b_i(v_{i-1}-v_i)  \right)\\
 &\le& v_{n-1}+\frac{1}{b_n}\left(\sum_{i=2}^{N_n}b_{N_n}v_1+\sum_{i=N_n+1}^{n-1}b_{n-1}(v_{i-1}-v_i)\right)\\
 &\le& v_{n-1}+\frac{N_nb_{N_n}}{b_n}v_1+\frac{b_{n-1}}{b_n}v_{N_n}
 \le v_{n-1}+\frac{1}{N_n}v_1+v_{N_n}\downarrow 0.\quad \qed
 \end{eqnarray*}
\end{proof}

In \cite[lemma 3.14]{IsraelRef}, this result was proved for the case of $b_n=n, n\in\N$.

\begin{lem}[Kronecker's Lemma]\label{LemKron}
	Let $(x_{n})$ be a summable sequence of elements in an Archimedean Riesz space $E$. Let $(b_{n})_{n\in\mathbb{N}}$ be a non-decreasing sequence of non-negative real numbers divergent to $+\infty$. Then $\displaystyle{\frac{1}{b_{n}}\sum_{i=1}^{n}b_{i}x_{i}\to 0}$ in order as $n\to \infty$.
\end{lem}

\begin{proof}
 Let $\displaystyle{s_n:=\sum_{i=n+1}^\infty x_i, n=0,1,2,\dots}$, then $s_n\to 0$ in order and
 \begin{eqnarray*}
 \left|\frac{1}{b_n}\sum_{i=1}^n b_ix_i\right| =\frac{1}{b_n}\left|b_ns_n-b_1s_0-\sum_{i=1}^{n-1} (b_{i+1}-b_i)s_i\right|
 \le |s_n|+\frac{b_1}{b_n}|s_0|+\frac{1}{b_n}\sum_{i=1}^{n-1} (b_{i+1}-b_i)|s_i|
 \end{eqnarray*}
 which converges to zero in order by Lemma \ref{CesaroConvergence}.
\qed 
\end{proof}


\section{Inequalities}

The inequalities presented in this section form the foundation on which much of the rest of this paper is based.

Taking the product Riesz space $\mathcal{K}=[L^{1}(T)]^n$ with componentwise ordering and defining $\mathbb{F}(x_i)_{i=1}^n=(\frac{1}{n}\sum_{j=1}^n x_j)_{i=1}^n$ 
we have that $\mathbb{F}$ is a conditional expectation operator on $\mathcal{K}$. Hence from \cite[Corollary 6.4]{JensensGrobler} or \cite[Theorem 3.7]{LpSpacesAz} we have the following theorem.

\begin{thm}[H\"older's inequality for sums]\label{ThmHolders}
	Let $T$ be a conditional expectation with natural domain $L^1(T)$ and $1\le p,q\le \infty$ with $\frac{1}{p}+\frac{1}{q}=1$ and $p=1$ for $q=\infty$. Let $n\in\mathbb{N}$. If $x_i\in L^p(T)$ and $y_i\in L^q(T)$ for all $i\in\{1,2,...,n\}$, then $x_{i}y_i \in L^1(T)$ for each $i$, and 
	$$\sum^{n}_{i=1}T|x_{i}y_{i}|\leq\left(\sum^{n}_{i=1}T|x_{i}|^{p}\right)^{\frac{1}{p}}\left(\sum^{n}_{i=1}T|y_{i}|^{q}\right)^{\frac{1}{q}}.$$
\end{thm}

From \cite[page 809]{LpSpacesAz} we have that
	$$|x+y|^{p}+|x-y|^{p}\leq2^p(|x|^{p}+|y|^{p})$$
for $1<p<\infty$ with $x,y\in E^u$. This inequality, however, is inadequate for our purposes and we require a refined version, i.e. the Clarkson's inequalities for $1<p<2$.
To this end we follow the approach of Ramaswamy \cite{Ramaswamy}.

\begin{thm}[Clarkson's Inequality]\label{FinallyProvedInequality}
	Let $E$ be a Dedekind complete Riesz space with weak order unit, say $e$ which we take as the multiplicative unit in the $f$-algebra $E^u$.
        For $x,y\in E^u$ and $1\le p\le 2$ we have
     \begin{equation}\label{clarkson}
       |x+y|^{p}+|x-y|^{p}\leq 2(|x|^{p}+|y|^{p}).
     \end{equation}
\end{thm}
\begin{proof}
For $p=1$ the result follows from the triangle inequality while
for $p=2$ the result follows from $|f|^2=f^2$, so we now 
consider only $1<p<2$.
 Taking $g(X)=X^{2/p}$ and $\mathbb{F}(X,Y)=(\frac{1}{2}(X+Y),\frac{1}{2}(X+Y))$ in Jensen's inequality of \cite{JensensGrobler} on the Riesz space 
  $F:=E^u\times E^u$ with componentwise ordering, we have that $\mathbb{F}(g(|a|^p,|b|^p))\ge g(\mathbb{F}(|a|^p,|b|^p))$ so
  $|a|^2+|b|^2\ge 2^{(p-2)/p}(|a|^p+|b|^p)^{2/p}$ and $2^{(2-p)/2}(|a|^2+|b|^2)^{p/2}\ge |a|^p+|b|^p$. Setting $a=x+y$ and $b=x-y$ we have
  \begin{equation}\label{clarkson-1}
    |x+y|^{p}+|x-y|^{p}\leq 2^{(2-p)/2}((x+y)^2+(x-y)^2)^{p/2}=2(x^2+y^2)^{p/2}.
  \end{equation}
  We now apply the $\infty$ case of H\"older's inequality of \cite{JensensGrobler} on the space $F$ with conditional expectation $\mathbb{F}$ as above to get
  \begin{equation}\label{clarkson-2}
    \mathbb{F}((|x|^p,|y|^p)(|x|^{2-p},|y|^{2-p}))\le \mathbb{F}((|x|^p,|y|^p))(|x|^{2-p}\vee |y|^{2-p},|x|^{2-p}\vee |y|^{2-p}).
  \end{equation}
   Here $|x|^{2-p}\vee |y|^{2-p}=(|x|^p\vee |y|^p)^{(2-p)/p}\le (|x|^p + |y|^p)^{(2-p)/p}$ by the commutation of multiplication and band projections in the $f$-algebra $E^u$.
   Hence from (\ref{clarkson-2}) we get
   $$\frac{1}{2}(x^2+y^2)\le \frac{1}{2} (|x|^p+|y|^p)(|x|^p + |y|^p)^{(2-p)/p}=\frac{1}{2} (|x|^p+|y|^p)^{2/p},$$
  which when combined with (\ref{clarkson-1}) gives (\ref{clarkson}).
\qed
\end{proof}

The strong law of large numbers for $p>2$ will make use of Riesz space versions of Burkholder's inequality, \cite[Theorem 16]{BurkholdersAz}, which we give here for completeness.
\begin{thm}[Burkholder's inequality]\label{ThmBurkholders}
 For $1< p<\infty$, there are constants $c_p, C_p>0$ such that
 $$C_pT|X_n|^p \leq T|S_n^\frac{1}{2}|^p \leq c_p T|X_n|^p,$$
 for each $(X_{n},T_n)_{n\in\mathbb{N}}$ a martingale in $L^p(T)$ compatible with $T$, i.e. $TT_n=T=T_nT$, for all $\in\N$.
 Here $\displaystyle{S_n:=\sum^{n}_{i=1}(X_{i}-X_{i-1})^2}$ and $X_0:=0$.
\end{thm}


\newsection{H\'ajek-R\'enyi-Chow maximal inequality}

We now recall some well known results regarding band projections on a Dedekind complete Riesz space, $E$, with a weak order unit, say $e$.
If $g\in E_+$ we denote the band projection onto the band generated by $g$ by $P_g$. In this setting every band is a principal band 
and if $B$ is a band in $E$ with band projection $Q$ onto $E$ then a generator of the band is $Qe$. Moreover for $f\in E_+$ we have
$\displaystyle{P_gf=\sup_{n\in\N} (f\wedge (ng))}$, see \cite[Theorem 11.5]{ZaanenBook}. Further if $(f_n)$ is a sequence in $E_+$ then
\begin{eqnarray}\label{LemBandProjLeftCts}
 \bigvee_{n=1}^{\infty}P_{f_n}=P_{\bigvee_{n=1}^{\infty}f_n}
\end{eqnarray}
since
 $$P_{\bigvee_{n=1}^{\infty}f_n}e = \bigvee_{m=1}^\infty\left(e\wedge m\left(\bigvee_{n=1}^{\infty}f_n\right)\right)
	  	 = \bigvee_{m,n=1}^{\infty}(e\wedge m f_n)= \bigvee_{n=1}^{\infty}P_{f_n}e.$$
We note however that for the case of infima only the following inequality can be assured 
\begin{equation}\label{InterestingOnInfSide}
 \bigwedge_{n=1}^{\infty}P_{f_n}\ge P_{\bigwedge_{n=1}^{\infty}f_n}.
\end{equation}
For reference we note that if $(g_n)$ is a sequence in $E$ then
$0\le P_{\bigwedge_{n=1}^{\infty}g_n^-}g_m^{+}\le P_{g_m^-}g_m^{+}=0$ giving
$$(I-P_{\bigwedge_{n=1}^{\infty}g_n^-})\bigvee_{m=1}^{\infty}g_m^{+}=\bigvee_{m=1}^{\infty}(I-P_{\bigwedge_{n=1}^{\infty}g_n^-})g_m^{+}=
\bigvee_{m=1}^{\infty}g_m^{+}.$$
Hence	 	
\begin{equation}\label{LemmaExclEqualsEqn}
 P_{\bigvee_{n=1}^{\infty}g_n^{+}}\leq I-P_{\bigwedge_{n=1}^{\infty}g_n^-}.
\end{equation}

Using telescoping series we generalise \cite[Lemma (6.1.1)]{EAndSBook} to vector lattices.
\begin{lem}\label{LemmaFirst} 
	Let $E$ be a Dedekind complete Riesz space weak order unit $e$. Let $(X_{i})\subset E$ be a sequence in $E$ and $g\in E$.
 Let $P_{i}:=P_{(g-X_{i})^{+}}, i\in\N,$ be the band projection onto the band generated by $(g-X_{i})^{+}$,
 then
\begin{eqnarray} (I-Q_n)g\leq X_{1}+\sum^{n-1}_{i=1}[Q_{i}(X_{i+1}-X_{i})] - Q_{n}X_{n},
 \label{pre-RHC}
\end{eqnarray}
where $\displaystyle{Q_{n}:=\prod_{j=1}^n P_j=P_{\left(g-\bigvee_{j=1}^n X_{j}\right)^{+}}}, n\in\N.$
\end{lem}

\begin{proof}
  Let $Q_{0}:=I$. From the definition of $P_i$ we have $P_j(g-X_j)=(g-X_j)^+$ and thus
  $(I-P_j)(g-X_j)=-(g-X_j)^-\le 0$. However $Q_{j-1}-Q_j=Q_{j-1}(I-P_j)$, so applying $Q_{j-1}$ to both sides of 
  $(I-P_j)(g-X_j)\le 0$, gives
  $(Q_{j-1}-Q_j)(g-X_j)\le 0$. Hence $(Q_{j-1}-Q_j)g\le (Q_{j-1}-Q_j)X_j$, which when summed over $j=1,\dots,n$ gives (\ref{pre-RHC}).
\qed
\end{proof} 

If $(f_i,T_i)$ is a submartingale in the Riesz space $E$ then so is $(f_i^+,T_i)$. To see this we observe that as $T_j$ a
positive operator and $f_j^+\ge f_j$ so
$T_if_j^+\ge T_if_j\ge f_i$ and as $f_j^+\ge 0$ so  $T_if_j^+\ge 0$, for $i\le j$. Hence $T_if_j^+\ge0\vee f_i=f_i^+$ for $i\le j$.

\begin{thm}[H\'ajek-R\'enyi-Chow maximal inequality]\label{MaximalIneq}
	Let $(Y_i,T_i)_{i\in\mathbb{N}}$ be a submartingale in $L^1(T)$. 
For $(a_{i})_{i\in\mathbb{N}}$ a non-decreasing sequence of positive real numbers and 
 $g\in {R}(T_{1})^{+}$ we have
	\begin{equation}\label{RHC-max}
		 T_{1}(I-U_{n})g\leq \dfrac{Y_{1}^{+}}{a_{1}}+\sum^{n-1}_{i=1}T_{1}\left[\dfrac{Y_{i+1}^{+}-Y_{i}^{+}}{a_{i+1}}\right]
	\end{equation}
where $U_{n}:=\prod_{i=1}^{n}P_{\left(g-\frac{Y_{i}}{a_{i}}\right)^{+}}=P_{\left(g-\bigvee_{i=1}^{n}\frac{Y_{i}}{a_{i}}\right)^{+}}$.
\end{thm}

\begin{proof}
 Let $Q=P_g$ be the band projection onto the band generated by $g$. Now as $g\in R(T_1)^+$ it follows that $Q$ and $T_1$ commute, see \cite[Theorem 3.2]{KLW2}.  As $(Y_i^+,T_i)$ is a submartingale, for $i\le j$, $T_iY_j^+\ge Y_i^+=T_i Y_i^+$, hence  
 \begin{equation}\label{sub-pos}
 T_i(Y_{j+1}^+-Y_j^+)\ge 0,
 \end{equation}
  and thus
\begin{equation}\label{RHC-max-1}
   (I-Q)T_{1}(I-U_{n})g=T_{1}(I-U_{n})(I-Q)g=0\leq (I-Q)\left(\dfrac{Y_{1}^{+}}{a_{1}}+\sum^{n-1}_{i=1}T_{1}\left[\dfrac{Y_{i+1}^{+}-Y_{i}^{+}}{a_{i+1}}\right]\right).
\end{equation}

Letting $X_i=Y_i^+/a_i, i\in\N,$ in Lemma \ref{LemmaFirst} we have, for $n\in\N$,
\begin{eqnarray} (I-Q_n)g\leq \frac{Y_{1}^+}{a_1}+\sum^{n-1}_{i=1}Q_{i}\left(\frac{Y^+_{i+1}}{a_{i+1}}-\frac{Y^+_{i}}{a_i}\right) - Q_{n}\frac{Y^+_{n}}{a_n}\le  \frac{Y_{1}^+}{a_1}+\sum^{n-1}_{i=1}Q_{i}\left(\frac{Y^+_{i+1}-Y^+_{i}}{a_{i+1}}\right)
 \label{pre-RHC-1}
\end{eqnarray}
where $\displaystyle{Q_{i}=P_{\left(g-\bigvee_{j=1}^i X_{j}\right)^{+}}}.$
Here $0\le Q_i\le I$ and $T_i(Y_{i+1}^+-Y_i^+)\ge 0$ so $Q_iT_i(Y_{i+1}^+-Y_i^+)\le T_i(Y_{i+1}^+-Y_i^+)$. Hence, as $T_1=T_1T_i$, from  (\ref{sub-pos}) and (\ref{pre-RHC-1}),
\begin{eqnarray} T_1(I-Q_n)g\leq
  \frac{Y_{1}^+}{a_1}+\sum^{n-1}_{i=1}T_1\left(\frac{Y^+_{i+1}-Y^+_{i}}{a_{i+1}}\right).
 \label{pre-RHC-2}
\end{eqnarray}
Since $g\ge 0$, we have 
$$g\wedge\left(\left(g-\frac{Y_i}{a_i}\right)\vee 0\right)=\left(g\wedge\left(g-\frac{Y_i}{a_i}\right)\right)\vee (g\wedge 0)=\left(g-\left(0\vee \frac{Y_i}{a_i}\right)\right)\vee  0,$$
giving $g\wedge \left(g-\frac{Y_i}{a_i}\right)^+=\left(g-\frac{Y_i^+}{a_i}\right)^+$, thus $QU_n=QQ_n$. 
Now, applying $Q$ to (\ref{pre-RHC-2}) and noting that $T_1Q=QT_1$ we have
\begin{eqnarray} 
 QT_1(I-U_n)g\leq
  Q\left(\frac{Y_{1}^+}{a_1}+\sum^{n-1}_{i=1}T_1\left(\frac{Y^+_{i+1}-Y^+_{i}}{a_{i+1}}\right)\right).
 \label{RHC-max-2}
\end{eqnarray}
Combining (\ref{RHC-max-1}) and (\ref{RHC-max-2}) gives (\ref{RHC-max}).
\qed
\end{proof}


\newsection{Submartingale convergence}

As an application of the H\'ajek-R\'enyi-Chow Maximal Inequality we give a weighted convergence theorem for submartingales, with a proof that is independent of upcrossing.

\begin{thm} [Submartingale convergence]\label{SubmartingaleConvergence}
 Let $p\geq 1$ and $(X_{i},T_{i})_{i\in\mathbb{N}}$
 be a non-negative submartingale in $L^p(T)$.
 Let $(a_{i})_{i\in\mathbb{N}}$ be a
 positive, non-decreasing, sequence of real numbers diverging to $\infty$. If
 \begin{equation}\label{PropositionBeforeStrongLaw}
  \sum^{\infty}_{i=1}T_{1}\left(\frac{X_{i+1}^{p}-X_{i}^{p}}{a_{i+1}^{p}}\right)
 \end{equation}
 converges in order, then $\frac{X_{n}}{a_{n}}$ tends to zero in order as $n$ tends to $\infty$.
\end{thm}
 
\begin{proof}
 By \cite[Corollary 4.5]{JensensGrobler}, $(X_{n}^{p},T_n)$ is a non-negative submartingale so 
 $Z_i:=(X_{i+1}^p-X_i^p)/a_{i+1}^p$ has $T_1Z_i\ge 0$ and by assumption $\sum_{i=1}^\infty T_1Z_i$ is order convergent, so by
 Lemma \ref{LemKron} 
\begin{equation}\label{AK}
 \frac{T_{1}X_{m+1}^{p}}{a_{m+1}^{p}}=\frac{1}{a_{m+1}^{p}}T_{1}\left(X_{1}^{p}+\sum^{m}_{i=1}(X_{i+1}^p-X_i^p)\right)=\frac{X_1^{p}}{a_{m+1}^{p}}+\frac{1}{a_{m+1}^{p}}\sum^{m}_{i=1}a_{i+1}^{p}T_{1}Z_i\to 0
\end{equation} in order as $m\rightarrow\infty$.

By (\ref{LemmaExclEqualsEqn}) and Theorem \ref{MaximalIneq} applied to $(X_i^p)_{i=m}^n$, with $g=te$, $t\in\R$ where $t>0$, for $n>m$, we have
\begin{equation}\label{ApplyHajek}
  T_{1}P_{\left(\bigvee_{i=m}^{n}\left(\frac{X_{i}^{p}}{a_{i}^{p}}-te\right)^{+}\right)}te
  \le T_{1}\left(I-P_{\bigwedge_{i=m}^{n}\left(\frac{X_{i}^{p}}{a_{i}^{p}}-te\right)^{-}}\right)te
 \le \dfrac{X_{m}^{p}}{a_{m}^{p}}+\sum^{n-1}_{i=m}T_{1}Z_i.
\end{equation}
Applying $T_1$ to (\ref{ApplyHajek}) and taking the order limit as $n\to\infty$, by (\ref{LemBandProjLeftCts}) we have
\begin{equation}\label{T1AndTakeLimitN}
 0\le t T_{1}\left(\bigvee_{i=m}^{\infty}P_{\left(\frac{X_{i}^{p}}{a_{i}^{p}}-te\right)^{+}}\right)e
 \leq T_{1}\left[\dfrac{X_{m}^{p}}{a_{m}^{p}}\right]+\sum^{\infty}_{i=m}T_{1}Z_i.
\end{equation}
Taking the order limit as $m\to \infty$ of (\ref{T1AndTakeLimitN}), by (\ref{AK}), we have
\begin{equation}\label{T1AndTakeLimitM}
 0\le t T_{1}\lim_{m\to\infty}\left(\bigvee_{i=m}^{\infty}P_{\left(\frac{X_{i}^{p}}{a_{i}^{p}}-te\right)^{+}}\right)e
 \leq \lim_{m\to\infty} T_{1}\left[\dfrac{X_{m}^{p}}{a_{m}^{p}}\right]+\lim_{m\to\infty}\sum^{\infty}_{i=m}T_{1}Z_i=0.
\end{equation}
Hence
$T_{1}\bigwedge_{m\in\N}\bigvee_{i=m}^{\infty}P_{\left(\frac{X_{i}^{p}}{a_{i}^{p}}-te\right)^{+}}e=0$ 
and by the strict positivity of $T_1$, 
$\bigwedge_{m\in\N}\bigvee_{i=m}^{\infty}P_{\left(\frac{X_{i}^{p}}{a_{i}^{p}}-te\right)^{+}}e=0$.
Now, by (\ref{LemBandProjLeftCts}) and (\ref{InterestingOnInfSide}),
$$0
\le P_{{\limsup_{i\to\infty}} \left(\frac{X_{i}^{p}}{a_{i}^{p}}-te\right)^{+}}e\le \bigwedge_{m\in\N}P_{\bigvee_{i=m}^{\infty}\left(\frac{X_{i}^{p}}{a_{i}^{p}}-te\right)^{+}}e
=\bigwedge_{m\in\N}\bigvee_{i=m}^{\infty}P_{\left(\frac{X_{i}^{p}}{a_{i}^{p}}-te\right)^{+}}e=0$$
and so
$\displaystyle{0\le \liminf_{i\to\infty} \frac{X_{i}^{p}}{a_{i}^{p}}\le \limsup_{i\to\infty} \frac{X_{i}^{p}}{a_{i}^{p}}\le te}$ for all $t>0$. Thus 
$\frac{X_{i}^{p}}{a_{i}^{p}}\to 0$ in order as $i\to\infty$.  	  	 
\qed
\end{proof}

{\bf Remark} Since $(X_{n}^{p},T_n)$ is a non-negative submartingale, by \cite[Corollary 4.5]{JensensGrobler}, taking $Y_{j+1}=\sum_{i=1}^j Z_i$ in the above theorem, we have that $(Y_j,T_j)$ is a $T_1$-bounded submartingale and Theorem \ref{SubmartingaleConvergence} follows directly from \cite[Theorem 3.5]{KLW3}.


\newsection{Chow's strong laws of large numbers}

We recall that $(Y_i,T_i)$ is a martingale difference sequence if $(T_i)$ is a filtration, $(Y_i)$ is adapted to $(T_i)$ and $T_{i-1}Y_i=0$ for $i\ge 2$.
In Theorem \ref{StrongLawPUnderTwo}, for $1\le p\le 2$, and Corollary \ref{p>2-cor} and Theorem \ref{InequalityForPOver2}, for $p>2$, Chow's strong law of large numbers is extended to martingale difference sequences in Riesz spaces.

\begin{thm}\label{StrongLawPUnderTwo}
	Let $1\le p\le 2$, and $(Y_{n},T_{n})_{n\in\mathbb{N}}$ be a martingale difference sequence in $L^{p}(T)$.
        Let $(a_{i})_{i\in\mathbb{N}}$ be a
	positive, non-decreasing sequence of real numbers divergent to infinity with
	\begin{equation}\label{HypothesisStrongLawPUnderTwo}
	\sum^{\infty}_{i=1}T_{1}\left(\frac{|Y_{i}|^{p}}{a_{i}^{p}}\right)
	\end{equation} 
       order convergent, then 
       $\displaystyle{\frac{1}{a_{n}}\sum_{i=1}^nY_i\to 0}$, in order, as $n$ tends to infinity.
\end{thm}

\begin{proof}
  Let $X_n=\displaystyle{\sum_{i=1}^nY_i}$ then $X_{i}+Y_{i+1}=X_{i+1}$ and  $X_{i}-Y_{i+1}=2X_{i}-X_{i+1}$ so Theorem \ref{FinallyProvedInequality} can be applied to give
 \begin{equation}\label{slln-1}
  |X_{i+1}|^p + |2X_{i}-X_{i+1}|^{p}\leq 2(|X_{i}|^p+|Y_{i+1}|^p).
 \end{equation}
Now as $(X_n,T_n)$ is a martingale, so $T_i(2X_i-X_{i+1})=X_i$ and by functional calculus, see \cite{JensensGrobler}, $|X_i|^p\in R(T_i)$ giving $T_i|X_i|^p=|X_i|^p$, hence
\begin{equation}\label{slln-2}
 T_i|X_i|^p=|T_i(2X_i-X_{i+1})|^p\leq T_i|2X_i-X_{i+1}|^p
\end{equation}
 where the final inequality follows from Jensen's inequality,  \cite[Theorem 4.4]{JensensGrobler}.
Combining (\ref{slln-1}) and (\ref{slln-2}) we have
 \begin{equation}\label{slln-3}
  T_i|X_{i+1}|^p - T_i|X_{i}|^p\le 2T_i|Y_{i+1}|^p.
 \end{equation}
By \cite[Corollary 4.5]{JensensGrobler}, $(|X_{i}|,T_i)$ and $(|X_{i}|^p,T_i)$ are submartingales so 
 \begin{equation}\label{slln-4}
  0\le \frac{T_i|X_{i+1}|^p - T_i|X_{i}|^p}{a_{i+1}^p}\le 2\frac{T_i|Y_{i+1}|^p}{a_{i+1}^p}
 \end{equation}
 which, with (\ref{HypothesisStrongLawPUnderTwo}), yields that $\displaystyle{\sum_{i=1}^\infty \frac{T_i|X_{i+1}|^p - T_i|X_{i}|^p}{a_{i+1}^p}}$ is order
convergent. The theorem now follows from Theorem \ref{SubmartingaleConvergence}.
\qed
\end{proof}
 
We can now bootstrap on Theorem \ref{StrongLawPUnderTwo} to obtain a strong law for $p>2$.

\begin{cor}\label{p>2-cor}
	Let $p> 2$ and $(Y_{n},T_{n})_{n\in\mathbb{N}}$ be a martingale difference sequence in $L^{p}(T_1)$.
        Let $(a_{i})_{i\in\mathbb{N}}$ be a
	positive, non-decreasing sequence of real numbers with $\displaystyle{\sum_{i=1}^\infty \frac{1}{a_i^k}}$ convergent in $\R$, and 
	$\displaystyle{
	\sum^{\infty}_{i=1}T_{1}\left(\frac{|Y_{i}|^{p}}{a_{i}^\gamma}\right)}$
       order convergent, where $p\ge \gamma+(\frac{p}{2}-1)k$, then 
       $\displaystyle{\frac{1}{a_{n}}\sum_{i=1}^nY_i\to 0},$ in order, as $n$ tends to infinity.
\end{cor}

\begin{proof}
 From H\"older's inequality, Theorem \ref{ThmHolders}, for $n>m$ we have
	$$\sum^{n}_{i=m}T_1\frac{|Y_{i}|^2}{a_i^2}\leq\left(\sum^{n}_{i=m}T_1\frac{|Y_{i}|^p}{a_i^\gamma}\right)^{\frac{2}{p}}\left(\sum^{n}_{i=m}\frac{e}{a_i^\delta}\right)^{1-\frac{2}{p}}$$
where $\delta=\frac{p-\gamma}{\frac{p}{2}-1}\ge k$ ensuring that
 $\displaystyle{\sum_{i=1}^\infty \frac{1}{a_i^\delta}}$ converges.
 Hence from Theorem \ref{StrongLawPUnderTwo} with $p=2$,
     $\displaystyle{\frac{1}{a_{n}}\sum_{i=1}^nY_i\to 0},$ in order, as $n$ tends to infinity.
\qed
\end{proof}

From Corollary \ref{p>2-cor}, if 
 $p> 2$ and $(Y_{n},T_{n})_{n\in\mathbb{N}}$ is a martingale difference sequence in $L^{p}(T_1)$ with
 	$\displaystyle{
	\sum^{\infty}_{i=1}T_{1}\left(\frac{|Y_{i}|^{p}}{i^{1+\frac{p}{2}-\delta}}\right)}$
       order convergent for some $\delta>0$  then 
       $\displaystyle{\frac{1}{n}\sum_{i=1}^nY_i\to 0}$, in order, as $n$ tends to infinity.
For this special case, of $a_i=i$, a more precise result can be given, as per \cite{ChowLLNPBig,ChowIneqAndLLN}. 

\begin{thm}\label{InequalityForPOver2}
	Let $p>2$ be a fixed number and let $(Y_{n},T_{n})_{n\in\mathbb{N}}$ be a martingale difference sequence in $L^{p}(T_1)$.  If $\displaystyle{\sum^{\infty}_{i=1}T_{1}\left(\frac{|Y_{i}|^{p}}{i^{1+\frac{p}{2}}}\right)}$ converges in order then $\displaystyle{\frac{1}{n}\sum^{n}_{i=1}Y_{i}\to 0}$ in order as $n\to\infty$.
\end{thm}

\begin{proof}
Let $\displaystyle{X_{n}=:\sum^{n}_{i=1}Y_{i}}$ for $n\in\mathbb{N}$, then, from Theorem \ref{SubmartingaleConvergence}, it suffices to prove the convergence as $n\rightarrow \infty$ of 
\begin{equation*}
 Z_n=\sum^{n}_{i=2}T_{1}\left(\frac{|X_{i}|^p-|X_{i-1}|^p}{i^p}\right)
 =\sum^{n-1}_{i=2}\left(\frac{1}{i^p}-\frac{1}{(i+1)^p}\right) T_{1}(|X_{i}|^p) + \frac{T_{1}(|X_{n}|^p)}{n^p}-\frac{T_{1}(|X_{1}|^p)}{2^p}.
\end{equation*}
	Since each term in the above summations is non-negative we need only show the boundedness of $Z_n, n\in\N$. 
 From Burkholder inequality, Theorem \ref{ThmBurkholders}, there is $C_p>0$ so that 
\begin{equation}\label{SLLN-burk}
  C_pT_1|X_n|^p\le T_1\left(\sum_{i=1}^n |Y_i|^2\right)^{p/2},
\end{equation}
for all $n\in\N$. Applying Jensen's inequality of \cite{JensensGrobler} we have
\begin{equation}\label{SLLN-jensen}
  T_1\left(\sum_{i=1}^n |Y_i|^2\right)^{p/2}\le \left(\sum_{i=1}^n T_1|Y_i|^2\right)^{p/2}.
\end{equation}
Now H\"older inequality, Theorem \ref{ThmHolders}, gives
\begin{equation}\label{SLLN-holder}
  \left(\sum_{i=1}^n T_1|Y_i|^2\right)^{p/2}\le n^{\frac{p}{2}-1}\sum_{i=1}^n T_1|Y_i|^p.
\end{equation}
Combining (\ref{SLLN-burk}), (\ref{SLLN-jensen}) and (\ref{SLLN-holder}) gives
\begin{equation}\label{SLLN-bound}
  \frac{T_1|X_n|^p}{n^p}\le \frac{1}{C_pn^{\frac{p}{2}+1}}\sum_{i=1}^n T_1|Y_i|^p.
\end{equation}
From Kronecker's Lemma, Theorem \ref{LemKron}, we have that
$\displaystyle{\frac{1}{n^{\frac{p}{2}+1}}\sum_{i=1}^n T_1|Y_i|^p\to 0}$
in order as $n\to \infty$. Thus the left hand side of (\ref{SLLN-bound}) is order bounded by say $h\in L^1(T_1)$ and 
\[\sum^{n-1}_{i=2}\left(\frac{1}{i^p}-\frac{1}{(i+1)^p}\right) T_{1}(|X_{i}|^p)\le p\sum^\infty_{i=1}\frac{1}{i^{p/2}} h,\]
giving $$Z_n\le p\sum^\infty_{i=1}\frac{1}{i^{p/2}} h +h-\frac{T_{1}(|X_{1}|^p)}{2^p}.$$ Here we have used that $n^{-p}-(n+1)^{-p}\le pn^{-p-1}, n\in\N$.
\qed
\end{proof}


\end{document}